\documentclass[11pt,draft]{article}

\usepackage{amssymb,amsmath,amsthm}
\usepackage{epsfig}
\usepackage{verbatim}
\usepackage{graphicx}

\usepackage{color}

\newtheorem{thm}{Theorem}[section]

\newtheorem{rem}[thm]{Remark}

\newtheorem{lemma}[thm]{Lemma}

\newcommand{\R}{\mathbb{R}}
\newcommand{\Q}{\Bbb{Q}}

\newcommand{\N}{\Bbb{N}}

\newcommand{\D}{\displaystyle}

\newcommand{\grad}{\nabla}

\newcommand{\F}{\mathcal{F}}
\newcommand{\gradp}{\grad^{\bot}}
\newcommand{\dx}{\partial_x}

\newcommand{\sgn}{{\rm sgn}\thinspace}

\newcommand{\al}{\alpha}

\newcommand{\ep}{\varepsilon}

\newcommand{\De}{\Delta_\al}

\newcommand{\eqdef}{\overset{\mbox{\tiny{def}}}{=}}

\begin{document}

\title{On the global existence for the Muskat problem}

\author{Peter Constantin, Diego C{\'o}rdoba, \\ Francisco Gancedo and Robert M. Strain}

\date{\today}

\maketitle

\begin{abstract}
The Muskat problem models the dynamics of the interface between two incompressible immiscible fluids with different constant densities.
In this work we prove three results.  First we prove an $L^2(\R)$ maximum principle, in the form of a new ``log''  conservation law \eqref{ln} which is satisfied by the equation \eqref{ec1d} for the interface.  
%The conservation law provides a very weak dissipative mechanism.  
Our second result is a proof of global existence of Lipschitz continuous solutions for initial data that satisfy 
$\|f_0\|_{L^\infty}<\infty$ and $\|\partial_x f_0\|_{L^\infty}<1$. We take advantage of the fact that the bound $\|\partial_x f_0\|_{L^\infty}<1$ is propagated by solutions, which grants strong compactness properties in comparison to the log conservation law.
Lastly, we prove a global existence result for unique strong solutions if the initial data is smaller than an explicitly computable constant, for instance $\| f\|_1 \le 1/5$. Previous results of this sort used a small constant $\epsilon \ll1$ which was not explicit \cite{Peter,SCH,DP,Esch2}.

\end{abstract}


{\bf Mathematics subjet classification: } 35A01, 76S05, 76B03

\setcounter{tocdepth}{1}
\tableofcontents

\section{Introduction}

The Muskat problem models the dynamics of an interface between two incompressible immiscible fluids with different characteristics, in porous media. The phenomena have been described using the experimental Darcy's law that is given in two dimensions by the following momentum equation:
\begin{equation}
\frac{\mu}{\kappa}u=-\nabla p-g(0,\rho).
\notag
\end{equation}
Here $\mu$ is viscosity, $\kappa$ permeability of the isotropic medium, $u$ velocity, $p$ pressure, $g$ gravity and $\rho$ density. Saffman and Taylor \cite{S-T} related this problem with the evolution of an interface in a Hele-Shaw cell since both physical scenarios can be modeled analogously (see also \cite{Peter} and reference therein).  Recently, the well-posedness has been shown without surface tension in \cite{ADP} (for previous work on the topic see \cite{Am},  \cite{Yi} and \cite{DP}) using arguments that rely upon the boundedness properties of the Hilbert transforms associated to $C^{1,\gamma}$ curves. Precise estimates are obtained with arguments involving conformal mappings, the Hopf maximum principle and Harnack inequalities. The initial data have to satisfy the Rayleigh-Taylor condition initially, otherwise the problem has been shown to be ill-posed  \cite{SCH}, \cite{DP}. With surface tension, the initial value problem becomes more regular, and instabilities do not appear \cite{Esch1}. The case of more than one free boundary has been treated in \cite{DP3} and \cite{Esch2}. 

 In this paper we consider an interface given by fluids of different constant densities $\rho^i$, with the same viscosity and without surface tension. The step function $\rho$ is  represented by
$$\rho(x,t)=\left\{\begin{array}{rl}
                  \rho^1,& x\in \Omega^1(t),\\
                   \rho^2,& x\in\Omega^2(t)=\R^2\setminus \Omega^1(t),
                \end{array}\right. $$
for $\Omega^i(t)$ connected regions. As the density $\rho$ is transported by the flow
$$
\rho_t+u\cdot\grad\rho=0,
$$
the free boundary evolves with the two dimensional velocity $u = (u_1, u_2)$. The Biot-Savart law recovers $u$ from the vorticity given by $\omega=\partial_{x_1}u_2-\partial_{x_2}u_1$, via the integral operator $$u(x,t)=\gradp \Delta^{-1}\omega(x,t).$$
Darcy's law then provides the relation $\omega=-\partial_{x_1}\rho$ where $\mu/\kappa$ and $g$ are taken equal to $1$ for the sake of simplicity. Then the velocity field can be obtained in terms of the density as follows:
$$
u(x,t)=PV\int_{\R^2}K(x-y)\rho(y,t)dy-\frac12(0,\rho(x,t)).
$$
Here the kernel $K$ is of Calder\'on-Zygmund type:
$$
K(x)=\frac{1}{\pi}\left(-\frac{x_1x_2}{|x|^2},\frac{x_1^2-x_2^2}{2|x|^2}\right),
$$
(see \cite{St3}). As a consequence of $\rho \in L^{\infty}(\R^2\times\R^+)$ it follows that the velocity belongs to $BMO$. Moreover, as $K$ is an even kernel, it has the  property that the mean of $K$ (in the principal value sense) are zero on hemispheres \cite{B-C}, and this yields a bound of the velocity $u(x,t)$ in terms of $C^{1,\gamma}$ norms ($0 <\gamma < 1$) of the free boundary \cite{DP3}.

 In order to have a well-posed problem we need to consider initially an interface parameterized as a graph of a function with the denser fluid below: $\rho^2>\rho^1$ as in  \cite{DP}. The interface is characterized as a graph of the function $(x,f(x,t))$. This characterization is preserved by the system and $f$ satisfies
\begin{align}
\begin{split}\label{ec1d}
\D f_t(x,t)&=\frac{\rho^2\!-\!\rho^1}{2\pi}PV\int_{\R} ~ d\al  ~
\frac{(\partial_x f(x,t)-\partial_x
f(x-\al,t))\al}{\al^2+(f(x,t)-f(x-\al,t))^2},
\\
f(x,0)&=f_0(x),
\quad x\in\R.
\end{split}
\end{align}
 The above equation can be linearized around the flat solution to find the following nonlocal partial differential equation
\begin{align}
\begin{split}\label{le}
f_t(x,t)&=-\frac{\rho^2-\rho^1}{2}\Lambda f(x,t),\\
f(\al,0)&=f_0(\al),\quad \al\in\R,
\end{split}
\end{align}
where the operator $\Lambda$ is the square root of the Laplacian. This linearization shows the parabolic character of the problem in the stable case ($\rho^2 > \rho^1$), as well as the ill-posedness in the unstable case ($\rho^2 < \rho^1$).

 The nonlinear equation (\ref{ec1d}) is ill-posed in the unstable situation and locally well-posed in $H^k$ ($k\geq 3$) for the stable case \cite{DP}. Furthermore the stable system gives a maximum principle $\|f\|_{L^{\infty}}(t)\leq \|f\|_{L^{\infty}}(0)$, see \cite{DP2}; decay rates are obtained for the periodic case as: 
 $$\|f\|_{L^\infty}(t)\leq \|f_0\|_{L^\infty}e^{-Ct},
 $$
and also for the case on the real line (flat at infinity) as: 
$$
\D\|f\|_{L^\infty}(t)\leq \frac{\|f_0\|_{L^\infty}}{1+Ct}.
$$ 
Numerical solutions performed in \cite{DPR} further indicate a regularizing effect. The decay
of the slope and the curvature is stronger than the rate of decay of the maximum of the difference between $f$ and its mean value. Thus, the irregular
regions in the graph are rapidly smoothed and the flat regions are smoothly bent. It is shown analytically in \cite{DP2} that, if the initial data satisfy $\|\partial_{x}f_0\|_{L^\infty}< 1$, then there is a maximum principle that shows that this derivative remains in absolute value smaller than 1. 

  The three main results we present in this paper are the following:\\
  
  1) In Section \ref{sec:L2}, we prove that
  a solution of \eqref{ec1d} satisfies
\begin{multline}\label{ln}
\|f\|^2_{L^2}(t)+\frac{\rho^2\!-\!\rho^1}{2\pi}\!\int_0^t  ds ~\int_\R d\al ~\int_\R dx~
\ln \left(1+\Big(\frac{f(x,s)\!-\!f(\al,s)}{x-\al}\Big)^2\right)  
\\
=\|f_0\|^2_{L^2}.
\end{multline}
Furthermore, we have the the inequality
$$
\int_\R dx ~ 
\int_\R d\al ~
\ln \left(1+\Big(\frac{f(x,s)\!-\!f(\al,s)}{x-\al}\Big)^2\right)
\leq 
C\|f\|_{L^1}(s).
$$
This identity shows a major difference with the linear equation \eqref{le} where the evolution of the $L^2$ norm provides a gain of half derivative for $\rho^2>\rho^1$:
\begin{equation}\label{lee}
\|f\|^2_{L^2}(t)+\left(\rho^2\!-\!\rho^1\right)\!\int_0^t ~ ds ~
\|\Lambda^{1/2} f\|_{L^2}^2(s) 
=\|f_0\|^2_{L^2},
\end{equation}
or equivalently
$$
\|f\|^2_{L^2}(t)+\frac{\rho^2\!-\!\rho^1}{2\pi}\!\int_0^t ds~ \int_\R  dx  ~ \int_\R  d\al  ~
\left(\frac{f(x,s)\!-\!f(\al,s)}{x-\al}\right)^2
=\|f_0\|^2_{L^2}.
$$
Notice that this linear energy balance \eqref{lee} directly implies compactness, whereas compactness does not follow from the nonlinear $L^2$ energy \eqref{ln}.\\

 2) In Section \ref{sec:init1} we prove global in time existence of Lipschitz continuous solutions in the stable case. We understand the solution of \eqref{ec1d} using its weak formulation:
\begin{align}
\begin{split}\label{sol:wf}
\int_0^T dt ~ \int_{\R} dx ~\eta_t(x,t)f(x,t)&
+
\int_{\R} dx ~ \eta(x,0)f_0(x)
\\
=
\int_0^T dt ~ \int_{\R} dx ~
\eta_x(x,t)\frac{\rho^2\!-\!\rho^1}{2\pi}
&
PV\int_{\R} d\al ~  \arctan\left(\frac{f(x,t)-f(\al,t)}{x-\al}\right).
\end{split}
\end{align} 
This equality  holds 
$\forall \eta\in C^\infty_c([0,T)\times\R)$.  For initial data $f_0$ satisfying $\|f_0\|_{L^\infty}<\infty$ and $\|\partial_x f_0\|_{L^\infty}<1$ we prove that
there exists a solution of \eqref{sol:wf} that remains in the spaces $f(x,t)\in C([0,T]\times\R)\cap L^\infty([0,T]; W^{1,\infty}(\R))$ for any $T>0$. We point out that, because of the condition $f\in L^\infty(\R)$, the nonlinear term in \eqref{sol:wf} has to be understood as a principal value for the integral of two functions, one in $\mathcal{H}^1$ and the other in $BMO$ \cite{St3}.

There are several results of global existence for small initial data (small compared to $1$ or $\epsilon \ll 1$) in several norms (more regular than Lipschitz) \cite{Peter,SCH,DP,Esch2} taking advantage of the parabolic character of the equation for small initial data. Here we show that we just need $\|\dx f_0\|_{L^\infty}<1$, therefore
$$
\Big|\frac{f_0(x)-f_0(\al)}{x-\al}\Big|<1.
$$
Notice that considering the first order term in the Taylor series of $\ln(1+y^2)$ for $|y|<1$, then the identity \eqref{ln} becames \eqref{lee}. \\
 
 3) Our third result, discussed in Section \ref{sec:15}, proves global existence of unique $C([0,T];H^3(\R))$ solutions if initially the norm \eqref{norm:s} of $f_0$ is controlled as $\|f_0\|_1< c_0$ where
$$
\|f_0\|_{1}=\int_{\R}d\xi ~ |\xi| |\hat{f}_0(\xi)|.
$$    The key point here, in comparison to previous work \cite{Peter,SCH,DP,Esch2}, is that the constant $c_0$ can be easily explicitly computed.  We have checked numerically that $c_0$ is not that small; it is greater than $1/5$.

\section{$L^2$ maximum principle}\label{sec:L2}

In this section we provide  a proof of the identity \eqref{ln}. 
As we are in the stable case, we take without loss of generality $(\rho^2-\rho^1)/(2\pi)=1$ to simplify the exposition. The contour equation \eqref{ec1d} can be written as follows:
\begin{align*}
\begin{split}
\D f_t(x,t)&=PV\int_{\R}\partial_x \arctan\Big(\frac{f(x,t)-
f(x-\al,t)}{\al}\Big)d\al.
\end{split}
\end{align*}
We multiply by $f$, integrate over $dx$, and use integration by parts to observe
\begin{align*}
\begin{split}
\frac{1}{2}\frac{d}{dt}\|f\|^2_{L^2}(t)&=
-\int_\R  dx  ~ \int_\R d\al ~   
f_x(x)\arctan\Big(\frac{f(x,t)-f(x-\al,t)}{\al}\Big) 
\\
&=
-\int_\R  dx  ~ \int_\R d\al ~   
f_x(x)\arctan\Big(\frac{f(x,t)-f(\al,t)}{x-\al}\Big).
\end{split}
\end{align*}
We use the splitting
\begin{align*}
\begin{split}
\frac{1}{2}\frac{d}{dt}\|f\|^2_{L^2}(t)&=-\int_\R\!\int_\R\Big(\!\frac{f(x,t)\!-\!f(\al,t)}{x-\al}\!\Big)\arctan\Big(\!\frac{f(x,t)\!-\!f(\al,t)}{x-\al}\!\Big)dx
d\al\\
-\int_\R\!\int_\R
\Big(&\!\frac{f_x(x)(x\!-\!\al)\!-\!(f(x,t)\!-\!f(\al,t))}{x-\al}\!\Big)\arctan\Big(\!\frac{f(x,t)\!-\!f(\al,t)}{x-\al}\!\Big)dx
d\al \\
&=I_1+I_2.
\end{split}
\end{align*}
We also use the function $G$ defined by 
$$
G(x)=x\arctan x-\ln\sqrt{1+x^2}=\int_0^x dy ~ \arctan y.
$$
With these, it is easy to observe that
$$
I_2=-\int_\R\int_\R (x-\al)\partial_xG\Big(\frac{f(x,t)-f(\al,t)}{x-\al}\Big) ~ dx d\al.
$$
The identity below
$$
\lim_{|x|\rightarrow \infty}(x-\al)G\Big(\frac{f(x,t)-f(\al,t)}{x-\al}\Big)=0,
$$
allows us to integrate by parts to obtain
\begin{align*}
\begin{split}
I_2&=\int_\R\int_\R G\Big(\frac{f(x,t)-f(\al,t)}{x-\al}\Big)dx d\al\\
&=-I_1-\int_\R\int_\R \ln\sqrt{1+\Big(\frac{f(x,t)-f(\al,t)}{x-\al}\Big)^2}dx d\al.\\
\end{split}
\end{align*}
This equality gives
\begin{align*}
\begin{split}
\frac{1}{2}\frac{d}{dt}\|f\|^2_{L^2}(t)&=-\int_\R\int_\R
\ln\sqrt{1+\Big(\frac{f(x,t)-f(\al,t)}{x-\al}\Big)^2}dx d\al,\\
\end{split}
\end{align*}
and integrating in time we get the desired identity.

 The above equality indicates that for large initial data, the system is not parabolic at the level of $f$. We prove below the inequality
$$
\int_\R\int_\R
\ln \Big(1+\Big(\frac{f(x,t)-f(\al,t)}{x-\al}\Big)^2\Big)dx d\al \leq 4\pi\sqrt{2}\|f\|_{L^1}(t)
$$
which shows that there is no gain of derivatives for the stable case. If the initial data are positive, then   $\|f\|_{L^1}(t)\leq \|f_0\|_{L^1}$ follows from \cite{DP2}, so that the dissipation is bounded in terms of the initial data with zero derivatives.

For the proof of the inequality, we denote by $J$ the integral
$$
J\eqdef \int_\R\int_\R
\ln\Big(1+\Big(\frac{f(x)-f(x-\al)}{\al}\Big)^2\Big) ~ dx d\al.
$$
We now use that the function $\ln(1+y^2)$ is increasing to observe that
$$
J\leq \int_\R\int_\R
\ln\Big(1+\frac{2|f(x)|^2}{\al^2}+\frac{2|f(x-\al)|^2}{\al^2}\Big) ~ dx d\al.
$$
The inequality $\ln(1+a^2+b^2)\leq \ln(1+a^2)+\ln(1+b^2)$ yields
$$
J\leq \int_\R\int_\R
\ln\Big(1+\frac{2|f(x)|^2}{\al^2}\Big)dx d\al+\int_\R\int_\R
\ln\Big(1+\frac{2|f(x-\al)|^2}{\al^2}\Big)dx d\al,
$$
and therefore
$$
J\leq 2\int_\R\int_\R
\ln\Big(1+\frac{2|f(x)|^2}{\al^2}\Big)dx d\al=K.
$$
For $K$ it is easy to get
$$
K=2\int_{\{x: |f(x)|\neq 0\}}dx \int_\R d\al ~ \ln\Big(1+\frac{2|f(x)|^2}{\al^2}\Big),
$$
so that an easy integration in $\al$ provides
$$
K=4\pi\sqrt{2}\int_{\{x: |f(x)|\neq 0\}}dx ~ |f(x)|=4\pi\sqrt{2}\|f\|_{L^1}.
$$
This concludes our discussion of the $L^2$ maximum principle \eqref{ln} for \eqref{ec1d}.

\section{A global existence result for data less than $1/5$}\label{sec:15}

In this section we prove global existence of $C([0,T];H^3(\R))$ small data solutions.  A key point is to consider the norm 
\begin{equation}
 \|f\|_s
 \eqdef
 \int_{\R}d\xi ~ |\xi|^s |\hat{f}(\xi)|,
 \quad 
 s \ge 1.
 \label{norm:s}
\end{equation}
 This norm allows us to use Fourier techniques for small initial data that give rise to a global existence result for classical solutions.

 \begin{thm}\label{15:thm}
 Suppose that initially $f_0 \in H^3(\R)$ and $\|f_0\|_1 < c_0$, where $c_0$ is a constant such that
$$
2\sum_{n\geq 1} (2n+1)^{2+\delta}c_0^{2n}\leq 1
$$ 
for $0<\delta<1/2$. Then there is a unique solution $f$ of \eqref{ec1d} that satisfies $f \in C([0,T];H^3(\R))$ 
for any $T>0$.\end{thm}

\begin{rem}
We compute the limit case $\delta=0$, so that
$$
2\sum_{n\geq 1} (2n+1)^{2}c_0^{2n}\leq 1
$$
for 
$$
0\leq c_0\leq\frac{1}{3}\sqrt{7-\frac{14\times5^{2/3}}{\sqrt[3]{9\sqrt{39}-38}}+2\sqrt[3]{5(9\sqrt{39}-38)}}\approx 0.2199617648835399.
$$
In particular, 
$$
2\sum_{n\geq 1}(2n+1)^{2.1}c_0^{2n} < 1,
$$
if say $c_0 \le 1/5$.
\end{rem}
 
 The remainder of this section is devoted to the proof of Theorem \ref{15:thm}.
 The contour equation for the stable Muskat problem \eqref{ec1d} can be written as
\begin{equation}
f_t(x,t)=-\rho (\Lambda f + T(f)),
\label{muskatEQ2d}
\end{equation}
where we recall that $\rho=\frac{\rho^2-\rho^1}{2}>0$ and we have
\begin{align}\label{td}
T(f)=\frac1{\pi}\int_\R\frac{\dx f(x)-\dx f(x-\al)}{\al}
\frac{\big(\frac{f(x)-f(x-\al)}{\al}\big)^2}{1+\big(\frac{f(x)-f(x-\al)}{\al}\big)^2}d\al,
\end{align}
We define
$$
\De f(x)
 \eqdef
\frac{f(x)-f(x-\al)}{\al}.
$$
We consider the evolution of the norm
$
\|f\|_1
$
\eqref{norm:s}:
\begin{align*}
\frac{d}{dt}\|f\|_1(t)&=\int_{\R}d\xi ~ |\xi| ~ \sgn (\hat{f}(\xi)) ~ \hat{f}_t(\xi)
\\
&=\rho \int_{\R}d\xi ~ |\xi| ~ \sgn (\hat{f}(\xi)) ~ (-|\xi|\hat{f}(\xi)-\F(T)(\xi)).
\end{align*}
We will show that the first term dominates the second term if initially $$\|f_0\|_1< \sqrt{(4-\sqrt{13})/6},\mbox{ where }\sqrt{(4-\sqrt{13})/6} >1/4.$$    The key point, again, is that the constant is given explicitly.

Notice that under the local existence theorem of \cite{DP}, 
this bound will be propagated for a short time. Then we may use the Taylor expansion
$$
\frac{x^2}{1+x^2}=\sum_{n=1}^\infty (-1)^{n+1} x^{2n},
$$
to obtain

\begin{equation}\label{tdt}
T(f)=\frac{-1}{\pi}\D\sum_{n\geq 1}(-1)^n\int_\R\dx (\De f)\,  (\De f)^{2n}d\al.
\end{equation}
Notice that
\begin{gather*}
\F(\De f)=\hat{f}(\xi) m(\xi,\al),\qquad \F(\dx \De f)=-i\xi \hat{f}(\xi) m(\xi,\al),
\\
m(\xi,\al)=\frac{1-e^{-i\xi\al}}{\al}.
\end{gather*}
Therefore
$$
\F(\dx (\De f)\,  (\De f)^{2n})=((-i\xi\hat{f} m)\ast (\hat{f} m)\ast \cdots \ast (\hat{f} m))(\xi,\al),
$$
with $2n$ convolutions, one with $-i\xi\hat{f} m$ and $2n-1$ with $\hat{f} m$.
Using \eqref{tdt}
\begin{align*}
\F(T)(\xi)
&=
\frac{i}{\pi}\D\sum_{n\geq 1}(-1)^n\int_\R d\al \int_\R d\xi_1\cdots\int_\R d\xi_{2n} (\xi\!-\!\xi_1)\hat{f}(\xi\!-\!\xi_{1})m(\xi\!-\!\xi_1,\al)
\\ 
\times \hat{f}(\xi_1\!&-\!\xi_2)m(\xi_1\!-\!\xi_2,\al)\cdots\hat{f}(\xi_{2n-1}\!-\!\xi_{2n})m(\xi_{2n-1}\!-\!\xi_{2n},\al)
\hat{f}(\xi_{2n})m(\xi_{2n},\al) 
\\
=\D\sum_{n\geq 1}
&
\int_\R d\xi_1\cdots\int_\R d\xi_{2n} (\xi\!-\!\xi_1)\hat{f}(\xi\!-\!\xi_{1})
\left(
\prod_{i=1}^{2n-1}  \hat{f}(\xi_i -\!\xi_{i+1}) 
\right)
%\hat{f}(\xi_1\!-\!\xi_2)\cdots\hat{f}(\xi_{2n-1}\!-\!\xi_{2n})
\hat{f}(\xi_{2n}) M_n,
\end{align*}
where $M_n=M_n(\xi,\xi_1,\ldots,\xi_{2n})$ is given by
\begin{align}
\label{Mul}
M_n\eqdef\frac{i}{\pi}(-1)^n\!\int_\R\! m(\xi\!-\!\xi_1,\al)
%m(\xi_1\!-\!\xi_2,\al)\cdots m(\xi_{2n-1}\!-\!\xi_{2n},\al)
\left(
\prod_{i=1}^{2n-1}  m(\xi_i\!-\!\xi_{i+1},\al)  
\right)
m(\xi_{2n},\al)d\al.
\end{align}
Since $m(\xi, \alpha ) = i \xi \int_0^1 ds~ e^{i \al (s-1) \xi}$
we obtain
\begin{align*}
M_n(\xi,\xi_1,\ldots,\xi_{2n})= m_n(\xi,\xi_1,\ldots,\xi_{2n}) \,(\xi_1-\xi_2)\cdots(\xi_{2n-1}-\xi_{2n})\xi_{2n},
\end{align*}
with
\begin{align*}
m_n&=\frac{i}{\pi}\int_0^1\!\!ds_1\cdots\int_0^1\!\!ds_{2n}\int_\R d\al ~
\frac{1-e^{-i\al(\xi-\xi_1) }}{\al}\\
&\qquad\qquad \times \exp\Big(i\al\sum_{j=1}^{2n-1}(s_j-1)(\xi_j-\xi_{j+1})+i\al(s_{2n}-1)\xi_{2n}\Big)
\\
&=\frac{i}{\pi}\int_0^1\!\!ds_1\cdots\int_0^1\!\!ds_{2n} \Big(PV\int_\R
\exp(i\al A)\frac{d\al}{\al}-PV\int_\R \exp(i\al B)\frac{d\al}{\al}\Big)\\
&=-\int_0^1ds_1\cdots\int_0^1ds_{2n}(\sgn A-\sgn B),
\end{align*}
and
\begin{align*}
A=\sum_{j=1}^{2n-1}(s_j-1)(\xi_j-\xi_{j+1})+(s_{2n}-1)\xi_{2n}=-\xi_1+\sum_{j=1}^{2n}s_j\xi_j -\sum_{j=1}^{2n-1}s_j\xi_{j+1}.
\end{align*}
Additionally
\begin{align*}
B&= -(\xi-\xi_1)+\sum_{j=1}^{2n-1}(s_j-1)(\xi_j-\xi_{j+1})+(s_{2n}-1)\xi_{2n}
\\
&=-\xi +\sum_{j=1}^{2n}s_j\xi_j-\sum_{j=1}^{2n-1}s_j\xi_{j+1}.
\end{align*}
It follows that
\begin{align*}
\F(T)(\xi)\!=\!\D\sum_{n\geq 1}&\int_\R d\xi_1\cdots\int_\R d\xi_{2n} ~
m_n(\xi,\xi_1,\ldots,\xi_{2n})\,(\xi-\xi_1)\hat{f}(\xi-\xi_{1})
\\
&\times
\left(
\prod_{i=1}^{2n-1}  (\xi_i -\!\xi_{i+1}) \hat{f}(\xi_i -\!\xi_{i+1}) 
\right)
%(\xi_1-\xi_2)\hat{f}(\xi_1-\xi_2)\cdots(\xi_{2n-1}-\xi_{2n})\hat{f}(\xi_{2n-1}-\xi_{2n})
\xi_{2n}\hat{f}(\xi_{2n}) ,
\end{align*}
with $|m_n(\xi,\xi_1,\ldots,\xi_{2n})|\leq2$.  
We then have
\begin{align*}
\int_{\R}d\xi ~ |\xi|&|\F(T)(\xi)|\leq 2 \D\sum_{n\geq 1}\int_{\R}d\xi\int_\R d\xi_1\cdots\int_\R d\xi_{2n} ~ 
|\xi|  |\xi-\xi_1||\hat{f}(\xi-\xi_{1})|
  \\
&\times  |\xi_1-\xi_2||\hat{f}(\xi_1-\xi_2)|\cdots|\xi_{2n-1}-\xi_{2n}||\hat{f}(\xi_{2n-1}-\xi_{2n})||\xi_{2n}||\hat{f}(\xi_{2n})|.
\end{align*}
The inequality $|\xi|\leq |\xi-\xi_1|+|\xi_1-\xi_2|+\cdots+|\xi_{2n-1}-\xi_{2n}|+|\xi_{2n}|$ yields
\begin{align*}
\int_{\R}d\xi |\xi||\F(T)(\xi)| &\leq 2\sum_{n\geq
1}(2n+1)\Big(\int_{\R}d\xi |\xi|^2 |\hat{f}(\xi)|\Big)\Big(\int_{\R}d\xi |\xi| |\hat{f}(\xi)|\Big)^{2n} ,
\end{align*}
and therefore
\begin{align*}
\int_{\R}d\xi |\xi||\F(T)(\xi)| &\leq  \Big(\int_{\R}d\xi |\xi|^2 |\hat{f}(\xi)|\Big)2\sum_{n\geq
1}(2n+1)\|f\|_1^{2n}\\
&\leq \Big(\int_{\R}d\xi |\xi|^2 |\hat{f}(\xi)|\Big) \frac{2\|f\|^2_1(3-\|f\|^2_1)}{(1-\|f\|^2_1)^2}.
\end{align*}
Notice  $\frac{2x^2(3-x^2)}{(1-x^2)^2}<1$ if $0\leq x< \sqrt{\frac{4-\sqrt{13}}{6}}\approx 0.256400964$.  If  $\|f_0\|_1<\sqrt{\frac{4-\sqrt{13}}{6}}$, then
 this inequality will continue to hold for some time so that
$$
\frac{d}{dt}\|f\|_1(t)\leq 0,
$$
and we conclude that $\|f\|_1(t) \le \|f_0\|_1$ if $\|f_0\|_1<\sqrt{\frac{4-\sqrt{13}}{6}}$.

Now we  repeat the argument but with $s>1$ in \eqref{norm:s}.
Our goal is to obtain
\begin{equation}\label{qsmp}
\frac{d}{dt}\|f\|_{2+\delta}(t)\leq 0, 
\quad
0<\delta<1/2.
\end{equation}
Let us point out that 
$$
\|f_0\|_{2+\delta}\leq C(\|f_0\|_{L^2}+\|\partial_x^3 f_0\|_{L^2})
$$
for $0<\delta<1/2$. Using the inequality
$$
|\xi|^{2+\delta}\leq (2n+1)^{1+\delta}(|\xi-\xi_1|^{2+\delta}+|\xi_1-\xi_2|^{2+\delta}+\cdots+
|\xi_{2n-1}-\xi_{2n}|^{2+\delta}+|\xi_{2n}|^{2+\delta}),
$$
we proceed as before to get
$$
\int_{\R}|\xi|^{2+\delta}|\F(T)(\xi)|d\xi\leq \int_{\R}|\xi|^{3+\delta}
|\hat{f}(\xi)| d\xi \, 2\sum_{n\geq 1}(2n+1)^{2+\delta}\|f\|_1^{2n}.
$$
In particular, taking $\|f\|_1$ small enough we find
$$
\int_{\R}|\xi|^{2+\delta}|\F(T)(\xi)|d\xi\leq \int_{\R}|\xi|^{3+\delta}
|\hat{f}(\xi)| d\xi,
$$
and bound \eqref{qsmp} therefore holds.

If $\|f\|_{C^{2,\delta}}$ remains bounded ($0<\delta<1$), then from previous work \cite{DP}, we can conclude that there is global existence in $C([0,T];H^3(\R))$ for any $T>0$. Since for 
$$
|g|_{C^\delta}=\sup_{y\neq 0}\frac{|g(x+y)-g(x)|}{|y|^{\delta}},
$$ 
we find
$$
\frac{|g(x+y)-g(x)|}{|y|^\delta}=\Big|\frac{C}{|y|^\delta}\int_\R \hat{g}(\xi)e^{ix\xi}(e^{iy\xi}-1)d\xi \Big|\leq C\int_\R |\xi|^\delta|\hat{g}(\xi)|d\xi,$$
and therefore
$$
\|f\|_{C^{2,\delta}}\leq C
\left(
\|f\|_{L^\infty}+\int_{\R} d\xi ~ |\xi||\hat{f}(\xi)|+ \int_\R d\xi ~ |\xi|^{2+\delta}|\hat{f}(\xi)|
\right).
$$
We conclude that the solution can be continued for all time if $\|f_0\|_1$ is initially smaller than a computable constant $c_0$ and
$
\|f_0\|_{2+\delta}
$
is bounded.  The constant $c_0$
  defined by the condition 
$$
2\sum_{n\geq 1}(2n+1)^{2+\delta}c_0^{2n} \leq 1,
$$
which has been numerically verified to be no smaller than say $1/5$.

\section{Global existence for initial data smaller than 1}\label{sec:init1}

We prove now the existence of a weak solution of the system \eqref{ec1d} which can be written as follows: 
\begin{equation}
f_t=\frac{\rho}{\pi}\partial_xPV\int_{\R}\arctan\left(\frac{f(x)-f(x-\al)}{\al}\right)d\al,
\label{muskatAT}
\end{equation}
where $\rho=(\rho^2-\rho^1)/2$.  We first extend the sense of the contour equation with a weak formulation: for any $\eta(x,t)\in C^\infty_c([0,T)\times\R)$, a weak solution $f$ should satisfy \eqref{sol:wf}.
%\begin{align}
%\begin{split}\label{weaksol}
%\int_0^T\!\!\int_{\R}\eta_t(x,t)f(x,t)dxdt&+\int_{\R}\eta(x,0)f_0(x)dx=\\
%\int_0^T\!\!\int_{\R} \eta_x(x,t)\frac{\rho}{\pi}&PV\int_{\R}\arctan\left(\frac{f(x,t)-f(\al,t)}{x-\al}\right)d\al dxdt.
%\end{split}
%\end{align}
We show here that this is the case if $\|\partial_x f_0\|_{L^\infty}<1$. 
Then it follows that  $\|f\|_{L^\infty}(t)\leq\|f_0\|_{L^\infty}$ and $\|\partial_x f\|_{L^\infty}(t)\leq\|\partial_x f_0\|_{L^\infty}<1$ as in \cite{DP2}. Then the solution is in fact Lipschitz continuous by Morrey's inequality. The main result we prove below is the following:

\begin{thm}\label{weakSOLthm}
Suppose that $\|f_0\|_{L^\infty}<\infty$ and $\|\partial_x f_0\|_{L^\infty}<1$.Then there exists a global in time weak solution of \eqref{sol:wf} that satisfies 
$$
f(x,t)\in C([0,T]\times\R)\cap L^\infty([0,T];W^{1,\infty}(\R)).
$$
In particular $f$ is Lipschitz continuous.
\end{thm}

The rest of this section is devoted to the proof of Theorem \ref{weakSOLthm}.
The first step is to prove global in time existence of classical solutions to the regularized model \eqref{regularizedM} below. This is done in Section \ref{sec:globalREG}.  Prior to that, in Section \ref{sec:apriori} we prove some necessary a priori bounds.  Then, in Section \ref{sec:approx} we explain how to approximate the initial data.  Section \ref{secweakSOL} shows how to prove the existence of solutions of \eqref{sol:wf}, subject to the strong convergence established in Section \ref{sec:strong}.

 From now, in the next two subsections we write $f=f^\varepsilon$ for the solution to \eqref{regularizedM} for the sake of simplicity of notation. The regularized model is given by
\begin{equation}
f_t(x,t)=-\ep C\Lambda^{1-\ep}f+\ep f_{xx}+\frac{\rho}{\pi}\partial_xPV\int_{\R}d\al ~ \arctan(\Delta_\al^\ep f(x)).
\label{regularizedM}
\end{equation}
where $C>0$ is an universal constant fixed below, and we define
$$
\Delta_\al^\ep f(x) \eqdef \frac{f(x)-f(x-\al)}{\phi(\al)},
$$
with $\phi(\al)=\phi^\ep(\al)=\al/|\al|^\ep$ and $\ep$ is small enough.
Initially we  consider the data $f_0\in W^{1,\infty}(\R)$ with
$\| \partial_x f_0\|_{L^\infty(\R)} < 1$.  We will explain how to approximate this initial data later on in Section \ref{sec:approx}.

\subsection{A priori bounds}\label{sec:apriori}

For the regularized system \eqref{regularizedM} we obtain the following two a priori bounds $$\|f\|_{L^\infty}\leq \|f_0\|_{L^\infty},\qquad \|\partial_x f\|_{L^\infty}\leq \|\partial_x f_0\|_{L^\infty}<1.$$
In order to prove the first one, we check the evolution of $$M(t)=\max_{x}f(x,t)=f(x_t,t).$$ 
Then, where $M$ is differentiable
\begin{multline*}
M'(t)=f_t(x_t,t)
\\
=-\ep C\Lambda^{1-\ep}f(x_t)+\ep f_{xx}(x_t)+\frac{\rho}{\pi}\partial_xPV\int_{\R}\arctan(\Delta_{\al}^\ep f({x_t}))d\al.
\end{multline*}
We have to deal with the third term above, as the first and the second ones have the correct sign. Now
\begin{equation}
I(x)=\partial_xPV\!\int_{\R}\arctan(\Delta_{\al}^\ep f({x}))d\al=\partial_xPV\!\int_{\R}\arctan(\Delta_{x-\al}^\ep f({x}))d\al,
\label{Idefin}
\end{equation}
and thus
\begin{align}
\begin{split}\label{nose}
I(x)&=\partial_xf(x)PV\!\int_{\R}\frac{\frac{1}{\phi(x-\al)}}{1+(\Delta_{x-\al}^\ep f({x}))^2}d\al\\
&\quad -(1-\ep)PV\!\int_{\R}\frac{\frac{f(x)-f(\al)}{|x-\al|^{2-\ep}}}{1+(\Delta_{x-\al}^\ep f({x}))^2}d\al.
\end{split}
\end{align}
Therefore $I(x_t)\leq 0$ (since $\partial_x f(x_t)=0$). Then $M'(t)\leq 0$ for a.e. $t\in(0,T]$ and $M(t)\leq M(0)$. Analogously $m(t)\geq m(0)$.

From \eqref{regularizedM} and \eqref{Idefin} we have
\begin{align*}
f_{xt}&=-\ep C\Lambda^{1-\ep}f_x+\ep f_{xxx}+\frac{\rho}{\pi}I_x.
\end{align*}
Using \eqref{nose} we rewrite $I(x)=J^1(x)+J^2(x)$ where
$$
J^1(x)=PV\int_{\R}\frac{f_x(x)(x-\al)-(f(x)-f(\al))}{|x-\al|^{2-\ep}}\frac{1}{1+(\Delta_{x-\al}^\ep f(x))^2}d\al,
$$
$$
J^2(x)=\ep PV\int_{\R}\frac{f(x)-f(x-\al)}{|\al|^{2-\ep}}\frac{1}{1+(\Delta_{\al}^\ep f(x))^2}d\al,
$$
to find
\begin{multline*}
J^1_x(x)=f_{xx}(x)PV\int_{\R}\frac{1}{\phi(x-\al)}\frac{1}{1+(\Delta_{x-\al}^\ep f(x))^2}d\al
\\
\quad-(2-\ep)PV\int_{\R}\frac{f_x(x)-\frac{f(x)-f(\al)}{x-\al}}{|x-\al|^{2-\ep}}\frac{1}{1+(\Delta_{x-\al}^\ep f(x))^2}d\al
\\
\quad-PV\int_{\R}\frac{f_x(x)(x-\al)-(f(x)-f(\al))}{|x-\al|^{2-\ep}}\frac{2\Delta_{x-\al}^\ep f(x)}{(1+(\Delta_{x-\al}^\ep f(x))^2)^2}
\\
\times
\frac{f_x(x)(x-\al)-(1\!-\!\ep)(f(x)-f(\al))}{|x-\al|^{2-\ep}}d\al,
\end{multline*}
and we split further $J^1_x(x)=K^1(x)+K^2(x)+K^3(x)+K^4(x)$ where
\begin{eqnarray*}
K^1(x) &= & f_{xx}(x)PV\int_{\R}\frac{1}{\phi(\al)}\frac{1}{1+(\Delta_{\al}^\ep f(x) )^2}d\al,
\\
K^2(x) &= & \ep PV\int_{\R}\frac{f_x(x)-\frac{f(x)-f(x-\al)}{\al}}{|\al|^{2-\ep}}\frac{1}{1+(\Delta_{\al}^\ep f(x))^2}d\al,
\\
K^3(x) &= & -PV\int_{\R}\frac{f_x(x)-\frac{f(x)-f(x-\al)}{\al}}{|\al|^{2-\ep}}
\frac{2}{1+(\Delta_{\al}^\ep f(x) )^2}d\al,
\end{eqnarray*}
and
\begin{multline*}
K^4(x)=-PV\int_{\R}d\al~ \frac{f_x(x)-\frac{f(x)-f(x\!-\!\al)}{\al}}{|\al|^{2-\ep}}\frac{2\Delta_{\al}^\ep f(x)}{(1+(\Delta_{\al}^\ep f(x))^2)^2}
\\
\times (f_x(x)|\al|^{\ep}-(1\!-\!\ep)\Delta_{\al}^\ep f(x)).
\end{multline*}
For $J^2$ it is easy to check that
\begin{align*}
J^2_x(x)&=\ep PV\int_{\R}\frac{f_x(x)-f_x(x-\al)}{|\al|^{2-\ep}}\frac{d\al}{1+(\Delta_{\al}^\ep f(x))^2}\\
&\quad -\ep PV\int_{\R}\frac{f_x(x)-f_x(x-\al)}{|\al|^{2-\ep}}\frac{2(\Delta_{\al}^\ep f(x))^2d\al}{(1+(\Delta_{\al}^\ep f(x))^2)^2}.
\end{align*}
Next, as we did before, we consider $M(t)=\max_x f_x=f_x(x_t,t)$. Where $M(t)$ is differentiable it follows
$$M'(t)=f_{xt}(x_t,t)=-\ep C\Lambda^{1-\ep}f_x(x_t)+\ep f_{xxx}(x_t)+\frac{\rho}{\pi}I_x(x_t).$$
Now we claim that if $M(t)<1$ then $M'(t)\leq 0$ for a.e.t. We can conclude analogously for $m(t)=\min_x f_x>-1$, $m'(t)\geq0$ for a.e.t.

We check that if $M(t)<1$, then
$$
-\ep C\Lambda^{1-\ep}f_x(x_t)+\ep f_{xxx}(x_t)+\frac{\rho}{\pi}I_x(x_t)\leq 0.
$$
We can use the following formulas for the operator $\Lambda^{1-\ep} f_x$
\begin{gather}
\Lambda^{1-\ep}f_x(x)=c_1(\ep)\int_\R\frac{f_x(x)-f_x(x-\al)}{|\al|^{2-\ep}}d\al,\label{L1}
\\
\Lambda^{1-\ep}f_x(x)=c_2(\ep)\int_\R\frac{f_x(x)-\frac{f(x)-f(x-\al)}{\al}}{|\al|^{2-\ep}}d\al,\label{L2}
\end{gather}
where $0<c_m\leq c_1(\ep),c_2(\ep)\leq c_M$ for $0\leq \ep\leq 1/4$. 

%It easy to get \eqref{L2} from \eqref{L1} using integrating by parts as follows:

%\begin{equation*}
%\Lambda^{1-\ep}(f_x)=c_1(\ep)\int_\R \frac{\partial_\al(f_x(x)(\al-x)+f(x)-f(\al))}{|x-\al|^{2-\ep}}d\al.
%\end{equation*}

We claim that
$$
-\ep C\Lambda^{1-\ep}f_x(x_t)+\frac{\rho}{\pi}(K^2(x_t)+J^2_x(x_t))\leq 0.
$$
We will show that
$$
-\ep \frac{C}2\Lambda^{1-\ep}f_x(x_t)+\frac{\rho}{\pi}K^2(x_t)\leq 0,
$$
using \eqref{L2} and that
$$
-\ep \frac{C}2\Lambda^{1-\ep}f_x(x_t)+\frac{\rho}{\pi}J^2(x_t)\leq 0,
$$
by \eqref{L1}. In fact
\begin{multline*}
-\ep \frac{C}2\Lambda^{1-\ep}f_x(x_t)+\frac{\rho}{\pi}K^2(x_t)
\\
=-\ep PV\int_{\R}\frac{f_x(x_t)-\frac{f(x_t)-f(x_t-\al)}{\al}}{|\al|^{2-\ep}}\frac{\frac{Cc_2(\ep)}2(\Delta_{\al}^\ep f(x_t))^2+\frac{Cc_2(\ep)}2-\frac\rho\pi}{1+(\Delta_{\al}^\ep f(x_t))^2}d\al.
\end{multline*}
The mean value theorem gives
\begin{equation*}
|f(x)-f(x-\al)|/|\al|\leq \|f_x\|_{L^\infty}.
\label{mvt}
\end{equation*}
Thus if we take $C\geq \frac{2\rho}{c_m\pi}$ we obtain the first inequality. Also
\begin{multline*}
-\ep \frac{C}2\Lambda^{1-\ep}f_x(x_t)+\frac{\rho}{\pi}J^2(x_t)=
\\
-\ep PV\int_{\R}\frac{f_x(x_t)-f_x(x_t-\al)}{|\al|^{2-\ep}}\frac{
\frac{Cc_1(\ep)}{2}(\Delta_{\al}^\ep f(x_t))^2+\frac{Cc_1(\ep)}2-\frac\rho\pi}{1+(\Delta_{\al}^\ep f(x_t))^2}d\al
\\
-\ep PV\int_{\R}\frac{f_x(x_t)-f_x(x_t-\al)}{|\al|^{2-\ep}}\frac{2(\Delta_{\al}^\ep f(x_t))^2d\al}{(1+(\Delta_{\al}^\ep f(x_t))^2)^2}\leq 0.
\end{multline*}

We find $f_{xxx}(x_t)\leq 0$ and $K^1(x_t)=0$. We still have to deal with $K^3$ and $K^4$. Considering $K^3(x_t)+K^4(x_t)$, we realize that if
$$
P(\al)=2+2(\Delta_\al^\ep f({x_t}) )^2+2(\Delta_\al^\ep f({x_t}))(f_x(x_t)|\al|^{\ep}-(1\!-\!\ep)\Delta_\al^\ep f({x_t}) )\geq 0,
$$
we are done. We rewrite
$$
P(\al)=2+2\ep (\Delta_\al^\ep f({x_t}) )^2+2(\Delta_\al^\ep f({x_t}) )f_x(x_t)|\al|^{\ep},
$$
and therefore we need
$$
|(\Delta_\al^\ep f({x_t}) ) f_x(x_t)|\al|^{\ep}|\leq 1.
$$
This fact holds if
$$
|f|_{C^{1-2\ep}}=\sup_{\al\neq0}\frac{|f(x_t)-f(x_t-\al)|}{|\al|^{1-2\ep}}<1.
$$
Now we will check that if $\|f\|_{L^\infty}\leq \|f_0\|_{L^\infty}$ and $\|f_x\|_{L^\infty}<1$ then $|f|_{C^{1-2\ep}}<1$ for $\ep$ small enough uniformly. We replace $2\ep$ by $\ep$ without lost of generality. If $\|f_0\|_{L^\infty}=0$ or $\|f_x\|_{L^\infty}=0$ all done.
Otherwise
$$
\frac{|f(x_t)-f(x_t-\al)|}{|\al|^{1-\ep}}\leq \|f_x\|_{L^\infty}\delta^{\ep},
$$
for $0<|\al|\leq \delta$ and
$$
\frac{|f(x_t)-f(x_t-\al)|}{|\al|^{1-\ep}}\leq 2\frac{\|f_0\|_{L^\infty}}{\delta^{1-\ep}},
$$
for $|\al|\geq \delta$. We take $\delta^{1-\ep}=2\|f_0\|_{L^\infty}/\|f_x\|_{L^\infty}$ and therefore
$$
|f|_{C^{1-\ep}}\leq \max\{\|f_x\|_{L^\infty},\|f_x\|_{L^\infty}^{1-\frac{\ep}{1-\ep}}(2\|f_0\|_{L^\infty})^\frac{\ep}{1-\ep}\}.
$$
Now it is clear that given $\|f_0\|_{L^\infty}$, if $\|f_x\|_{L^\infty}<1$ there exists $\ep_0>0$ such that, for any $0\leq \ep\leq \ep_0$ it follows that $|f|_{C^{1-\ep}}\leq 1$.
%Here notice that $\ep>0$ is small depending upon the size of $\|f_0\|_{L^\infty}$ which is allowed to be large.

\subsection{Global existence for the regularized model}\label{sec:globalREG}

We use here the a priori bounds to show global existence. Local existence can be easily proving using the local existence proof for the non-regularized Muskat problem \eqref{ec1d}, as in \cite{DP}. 
We use energy estimates and the Gronwall inequality. As we did for \eqref{ec1d}, it follows that
\begin{align*}
\frac{d}{dt}\|f\|_{L^2}(t)&=-\frac{\rho}{\pi}\int_\R\int_\R
\frac{1-\ep}{|x-\al|^\ep}\ln \Big(1+\Big(\frac{f(x,t)\!-\!f(\al,t)}{x-\al}\Big)^2\Big)dx d\al\\
&\quad -2C\ep\|\Lambda^{(1-\ep)/2} f\|_{L^2}(t)-2\ep\|f_x\|_{L^2}(t).
\end{align*}
Therefore
%\begin{equation}\label{L2mpr}
$
\|f\|_{L^2}(t)\leq \|f_0\|_{L^2}.
$
%\end{equation}

\begin{rem}
The theorem of this section can also be found with 
$$
\|f_0\|_{L^2}<\infty\quad \mbox{instead of} \quad \|f_0\|_{L^\infty}<\infty.
$$ 
We picked the version above because it is more general.  We see that if the solution satisfies initially a $L^2$ bound then
$
f(x,t)\in L^\infty([0,T];L^2(\R)).
$
\end{rem}

Next, we consider the evolution of
\begin{align*}
\int_\R \partial_x^3f\partial_x^3f_t dx&\leq -C\ep\|\Lambda^{(1-\ep)/2} \partial_x^3f\|_{L^2}(t)
-\ep\|\partial_x^4f_x\|^2_{L^2}+L_1+L_2,
\end{align*}
where
$$
L_1=\frac{\rho}{\pi}\int_{\R} \partial_x^3f(x)\partial_x^3\Big(PV\int_{\R}
\frac{f_x(x)-f_x(x-\al)}{\phi(\al)}d\al\Big) dx,
$$
$$
L_2=-\frac{\rho}{\pi} \int_{\R}\partial_x^3f(x)\partial_x^3\Big(PV\int_{\R} \frac{f_x(x)-f_x(x-\al)}{\phi(\al)}\frac{(\Delta_\al^\ep f(x))^2d\al}{1+(\Delta_\al^\ep f(x))^2}\Big) dx.
$$
The term $f_x(x)$ cancel out in $L_1$ due to the PV and an integration by parts shows that
$$
L_1=-\frac{\rho}{\pi}C(\ep)\int_{\R} \partial_x^3f(x)\Lambda^{1-\ep}\partial_x^3f(x)dx\leq 0.$$
For $L_2$ one finds
$$
L_2=\frac{\rho}{\pi} \int_{\R}\partial_x^4f(x)\partial_x^2\Big(PV\int_{\R} \frac{f_x(x)-f_x(x-\al)}{\phi(\al)}\frac{( \Delta_\al^\ep f(x) )^2d\al}{1+( \Delta_\al^\ep f(x) )^2}\Big) dx,
$$
and the splitting $L_2=M_1+M_2+M_3$ gives
$$
M_1=\frac{\rho}{\pi} \int_{\R}\partial_x^4f(x)\int_{\R} \frac{\partial_x^3f(x)-\partial_x^3f(x-\al)}{\phi(\al)}\frac{( \Delta_\al^\ep f(x) )^2d\al}{1+( \Delta_\al^\ep f(x) )^2} dx,
$$
\begin{align*}
M_2=&\frac{3\rho}{\pi} \int_{\R}\partial_x^4f(x)\int_{\R} \frac{\partial_x^2f(x)-\partial_x^2f(x-\al)}{\phi(\al)}\frac{f_x(x)-f_x(x-\al)}{\phi(\al)}\\
& \qquad\qquad\qquad\qquad\qquad\qquad\qquad\qquad\qquad\times\frac{2( \Delta_\al^\ep f(x) )d\al}{(1+( \Delta_\al^\ep f(x) )^2)^2} dx,
\end{align*}
$$
M_3=\frac{\rho}{\pi} \int_{\R}\partial_x^4f(x)\int_{\R} \Big(\frac{f_x(x)-f_x(x-\al)}{\phi(\al)}\Big)^3\frac{(2-6(\Delta_\al^\ep f(x) )^2)d\al}{(1+( \Delta_\al^\ep f(x) )^2)^3} dx.
$$

For $M_1$ we proceed as follows
\begin{multline*}
\left| M_1 \right| =\frac{\rho}{\pi} \left(\int_{|\al|>1} d\al\int_\R dx+\int_{|\al|<1} d\al\int_\R dx\right)
\\
\leq C(\ep)(\|f\|_{L^\infty}+1)\|\partial_x^3f\|_{L^2}\|\partial_x^4f\|_{L^2}.
\end{multline*}
The identity
$$
\partial_x^2f(x)-\partial_x^2f(x-\al)=\int_0^1\partial_x^3f(x+(s-1)\al) \al ds,
$$
yields
\begin{multline*}
\left| M_2 \right|
\leq
\frac{6\rho}{\pi} \int_0^1ds\int_{|\al|<1} \frac{d\al}{|\al|^{1-2\ep}}\int_\R dx
|\partial_x^4f(x)||\partial_x^3f(x+(s-1)\al)|
\\
\times(|f_x(x)|+|f_x(x-\al)|)
\\
+
\frac{6\rho}{\pi} \int_0^1\!\!ds\int_{|\al|>1}~ \frac{d\al}{|\al|^{2-3\ep}}\int_\R \!\!dx ~ |\partial_x^4f(x)||\partial_x^3f(x\!+\!(s\!-\!1)\al)|
\\
\times(|f_x(x)|\!+\!|f_x(x\!-\!\al)|)(|f(x)|\!+\!|f(x\!-\!\al)|),
\end{multline*}
and therefore
$$
\left| M_2 \right| \leq C(\ep)(1+\|f\|_{L^\infty})\|\partial_x^3f\|_{L^2}\|\partial_x^4f\|_{L^2}\|f_x\|_{L^\infty}.
$$

In $M_3$ we use the splitting $M_3=N_1+N_2$ where
$$
N_1=\frac{\rho}{\pi}\int_{|\al|>1} d\al\int_\R dx,\qquad
N_2=\int_{|\al|<1} d\al\int_\R dx,
$$
and then
\begin{align*}
\left| N_1 \right|
&\leq \frac{16\rho}{\pi}\|f_x\|^2_{L^\infty} \int_0^1ds\int_{|\al|>1} \frac{d\al}{|\al|^{3-3\ep}}\int_\R dx
|\partial_x^4f(x)|(|f_x(x)|+|f_x(x-\al)|) \\
&\leq C\|f_x\|^2_{L^\infty}\|f_x\|_{L^2}\|\partial_x^4f\|_{L^2}\\
&\leq C\|f_x\|^2_{L^\infty}(\|f\|_{L^2}+\|\partial_x^3f\|_{L^2})\|\partial_x^4f\|_{L^2}.
\end{align*}
To finish, the equality
$$
f_x(x)-f_x(x-\al)=\int_0^1\partial_x^2f(x+(s-1)\al)~ \al~ ds,
$$
allows us to obtain (since $\frac{1}{2}+\frac{1}{4}+\frac{1}{4} = 1$):
\begin{align*}
\left| N_2 \right|
&\leq \frac{16\rho}{\pi}\|f_x\|_{L^\infty}\!\!\int_0^1\!\!dr \int_0^1\!\!ds\int_{|\al|<1}~ \frac{d\al}{|\al|^{1-3\ep}}\int_\R dx\\
&\qquad\qquad\qquad\times|\partial_x^4f(x)|
|\partial_x^2f(x+(r-1)\al)||\partial_x^2f(x+(s-1)\al)| \\
&\leq C\|f_x\|_{L^\infty}\|\partial_x^4f\|_{L^2}\|\partial_x^2f\|_{L^4}^2.
\end{align*}
The following estimate
\begin{multline*}
\|\partial_x^2f\|_{L^4}^4=\int_{\R}(\partial_x^2f)^3\partial_x^2fdx=-3\int_{\R}(\partial_x^2f)^2\partial_x^3f\partial_xfdx
\\
\leq 3\|f_x\|_{L^\infty}\|\partial_x^2f\|_{L^4}^2\|\partial_x^3f\|_{L^2},
\end{multline*}
yields
$$
\left| N_2 \right| \leq C\|f_x\|^2_{L^\infty}\|\partial_x^4f\|_{L^2}\|\partial_x^3f\|_{L^2}.
$$
Using Young's inequality
\begin{align*}
\frac{d}{dt}\|\partial_x^3f\|^2_{L^2}\leq C(\ep)(\|f\|^4_{L^\infty}+\|f\|^2_{L^\infty}+\|f_x\|^4_{L^\infty}&+\|f_x\|^2_{L^\infty}+1)\\
&\times(\|f\|^2_{L^2}+\|\partial_x^3f\|^2_{L^2}),
\end{align*}
and therefore the Gronwall inequality yields
$$
\|f\|^2_{L^2}(t)+\|\partial_x^3f\|^2_{L^2}(t)\leq 
(\|f_0\|^2_{L^2}+\|\partial_x^3f_0\|^2_{L^2})\exp\Big(\int_0^t C(\ep)G(s)ds\Big),
$$
for 
$$
G(s)=\|f\|^4_{L^\infty}+\|f\|^2_{L^\infty}+\|f_x\|^4_{L^\infty}(s)+\|f_x\|^2_{L^\infty}(s)+1.
$$ We find $f\in C([0,T];H^3(\R))$ for any $T>0$ by the a priori bounds.

\subsection{Approximation of the initial data}
\label{sec:approx}

The approximation of the initial data described below is needed in order to construct a weak solution. First consider a common approximation to the identity $\zeta \in C^\infty_c(\R)$
 satisfying
$$
\int_{\R} dx ~ \zeta(x) = 1,
\quad
\zeta \ge 0,\quad \zeta(x)=\zeta(-x).
$$
Now we denote the standard mollifier $\zeta_\ep(x) = \zeta(x/\ep) /\ep$
so that $\zeta_\ep(x)$ continues to satisfy the normalization condition above.

For any $f_0\in W^{1,\infty}(\R)$ and $\|\partial_x f_0\|_{L^\infty}<1$, we define the initial data for the regularized system as follows
$$
f_0^\ep (x) = \frac{(\zeta_\ep * f_0)(x)}{1+\ep x^2}.
$$
Notice that $f_0^\ep \in H^s(\R)$ for any $s>0$,   and
$$
\|f_0^\ep\|_{L^\infty}\leq \|f_0\|_{L^\infty}.
$$
More importantly,
$
\|\partial_x f_0^\ep\|_{L^\infty}\leq \|\partial_xf_0\|_{L^\infty}
$
if $\ep$ is sufficiently small (here $\ep$ will generally depend upon the size of $\|f_0\|_{L^\infty}$).
In particular
$$
\partial_x f_0^\ep (x)= \frac{(\zeta_\ep * \partial_x f_0)(x)}{1+\ep x^2}
-2 \ep x
\frac{(\zeta_\ep * f_0)(x)}{(1+\ep x^2)^2},
$$
and clearly
$$
\left|  \frac{(\zeta_\ep * \partial_x f_0)(x)}{1+\ep x^2}
\right|
\le
\|
(\zeta_\ep * \partial_x f_0)
\|_{L^\infty(\R)}
\le
\|
\partial_x f_0
\|_{L^\infty(\R)}.
$$
On the other hand, by splitting into
$
\left| x \right| \le \ep^{-2/3}
$
and
$
\left| x \right| > \ep^{-2/3}
$
we have that
$$
2 \ep x (1+\ep x^2)^{-2}
\le 2\max\{\ep^{1/3}, \ep\}.
$$
On the unbounded region we have
$$
x(1+\ep x^2)^{-2}
=
\left(\frac{1}{\sqrt{x}}+\ep x^{3/2}\right)^{-2} \le 1.
$$
Thus, the desired bound follows if $\ep$ is small enough. 
Therefore global existence of the regularized system \eqref{regularizedM} is true with $f_0^\ep$ if $\ep$ is small enough.

Now consider the solution to the regularized system
\eqref{regularizedM}
with initial data given by the $f_0^\ep$ just described above.
For $\ep>0$ sufficiently small, we decompose
$$
\int_{\R}\eta(x,0)f_0^\ep(x)dx=\int_{\R}\eta(x,0)\frac{(\zeta_\ep * f_0)(x)}{1+\ep x^2}dx=I^\ep_1+I^\ep_2,
$$
where
$$
I^\ep_1=\int_{\R}\eta(x,0)(\zeta_\ep * f_0)(x)\Big(\frac{1}{1+\ep x^2}-1\Big)dx,
$$
and
$$
I^\ep_2=\int_{\R}\eta(x,0)(\zeta_\ep * f_0)(x)dx.
$$
We apply the dominated convergence theorem to find that as $\ep \downarrow 0$ it holds that $I^\ep_1\to 0$. For $I^\ep_2$ we write
$$
I^\ep_2=
\int_{\R}\zeta_\ep *(\eta(\cdot,0))f_0(x)dx.
$$
The $L^1$ approximation of the identity property shows that
$$
I^\ep_2
\to
\int_{\R}\eta(x,0)f_0(x)dx.
$$
Thus, it remains to check the convergence of the rest of the terms in \eqref{sol:wf}.

\subsection{Weak solution}\label{secweakSOL}

In this section we  prove that solutions of the regularized system converge to a weak solution satisfying the bounds
\begin{equation}
\|f\|_{L^\infty}(t)\leq\|f_0\|_{L^\infty},
\quad
\|\partial_x f\|_{L^\infty}(t)\leq\|\partial_x f_0\|_{L^\infty}<1.
\label{oneBOUND}
\end{equation}
Given a collection of regularized solutions $\{f^\ep\}$ to \eqref{regularizedM}, we have the uniform (in $\ep >0$) bound
\begin{equation}
\|f^\ep\|_{L^\infty}(t)\leq\|f_0\|_{L^\infty},
\quad \|\partial_x f^\ep\|_{L^\infty(\R)}(t) \leq 1,
\quad \ep > 0.
\label{unifEbd}
\end{equation}
This implies that there is a subsequence (denoted again by $f^\ep$) such that
$$
\int_0^T\int_{\R}f^\ep(x,t) g(x,t) dxdt\longrightarrow  \int_0^T\int_{\R}f(x,t) g(x,t) dxdt,
$$
$$
\int_0^T\int_{\R}\partial_xf^\ep(x,t) g(x,t) dxdt\longrightarrow  \int_0^T\int_{\R}\partial_x f(x,t) g(x,t) dxdt,
$$
for $f\in L^\infty([0,T];W^{1,\infty}(\R))$ and any $g\in L^1([0,T]\times\R)$ by the Banach-Alaoglu theorem. We find weak* convergence in $L^\infty([0,T];W^{1,\infty}(\R))$.

We denote $B_N=[-N,N]$, then we {\it claim} that there is a subsequence (denoted again by $f^\ep$) such that 
$$
\|f^\ep-f\|_{L^\infty([0,T]\times B_N)}\to 0,\quad \mbox{as}\quad\ep\to 0.
$$
We will prove this fact in the next Section \ref{sec:strong}. Then, up to a subsequence, we find uniform convergence of $f^\ep$ to $f$ on compact sets. Since $f^\ep \in C([0,T]\times \R)$ we find that $f$ is continuous.

The only thing to check is that as $\ep \downarrow 0$ we have
\begin{multline*}
\int_0^T ~ dt ~
\int_{\R} ~ dx ~
\eta_x(x,t)\frac{\rho}{\pi}PV\!\!\int_{\R}~ d\al~ \arctan\left(\frac{f^\ep(x)-f^\ep(x-\al)}{\phi^\ep(\al)}\right)
\\
\to
\int_0^T ~ dt ~
\int_{\R} ~ dx ~
\eta_x(x,t)\frac{\rho}{\pi}PV\!\!\int_{\R}~ d\al~ \arctan\left(\frac{f(x)-f(x-\al)}{\al}\right),
\end{multline*}
where $\phi^\ep(\al)=\al/|\al|^{\ep}$.
The rest of the terms will converge in the usual obvious way (since they are linear).

Choose $M>0$ so that  supp$(\eta)\subseteq B_M$. For any small $\delta >0$ and any large $R \gg 1$, with $R>M+1$ we split the integral as
$$
\int_{\R}~ d\al~
=
\int_{B_\delta}~ d\al~
+
\int_{B_R - B_\delta}~ d\al~
+
\int_{B_R^c}~ d\al.
$$
We begin with a proof that the first and last integrals separately are arbitrarily small independent of $\ep$ for $R>0$ sufficiently large and for $\delta>0$ sufficiently small.
One finds that
$$
\left|
\arctan\left(\frac{f^\ep(x)-f^\ep(x-\al)}{\phi^\ep(\al)}\right)
\right|\le \frac\pi2.
$$
Here we don't need any regularity for $f^\ep$, and we conclude
\begin{multline*}
\left|
\int_0^T dt 
\int_{\R} dx ~
\eta_x(x,t)\frac{\rho}{\pi}PV\int_{B_\delta}~ d\al~ \arctan\left(\frac{f^\ep(x)-f^\ep(x-\al)}{\phi^\ep(\al)}\right)
\right|
\\
\le \rho \| \eta_x\|_{L^1([0,T]\times \R)}\delta.
\end{multline*}
Therefore this term can clearly be chosen arbitrarily small, depending upon the smallness of $\delta$.

We now estimate the term integrated over $B_R^c$.  We note that
$$
\arctan y = \int_0^1\frac{d}{ds}(\arctan (sy)) ds=y \int_0^1\frac{1}{1+s^2y^2} ds,
$$
and therefore
$$
\arctan y = y\left(1- \int_0^1\frac{s^2y^2}{1+s^2y^2} ds\right).
$$
This is morally the first order Taylor expansion for $\arctan$ with remainder in integral form.
With this expression we have that
\begin{multline*}
PV\!\!\int_{B_R^c}~ d\al~ \arctan\left(\frac{f^\ep(x)-f^\ep(x-\al)}{\phi^\ep(\al)}\right)
=
-H^\ep_R(f^\ep)
\\
-PV\int_{B_R^c}~ d\al~ \left(\frac{f^\ep(x)-f^\ep(x-\al)}{\phi^\ep(\al)}\right)^3 \int_0^1\frac{s^2ds}{1+
s^2\left(\frac{f^\ep(x)-f^\ep(x-\al)}{\phi^\ep(\al)}\right)^2}.
\end{multline*}
Here $H^\ep_R$ is a (Hilbert-type) Transform, which has the form
$$
H^\ep_R(f^\ep) \eqdef PV\!\!\int_{B_R^c} d\al~\frac{f^\ep(x-\al)}{\phi^\ep(\al)}.
$$
The principal value is evaluated at infinity (if necessary).
 For the second term in the left hand side notice that the integral is over $B_R^c$ and  the principal value is not necessary.  In particular, we have
$$
\left|
\int_{B_R^c} d\al \int_0^1 ds
\right|
\le
C\| f^\ep \|^3_{L^\infty}
\int_{R}^\infty~ \frac{d\al}{\al^{3-3\ep}}\le
\frac{C\| f_0 \|^3_\infty}{R}.
$$
This term is therefore arbitrarily small if $R$ is chosen sufficiently large. We are going to show the same for
$$
I_R\eqdef \int_{-M}^M dx ~\eta_x(x,t)H^\ep_R(f^\ep).
$$
We write $I_R=J_R+K_R$ where
$$
J_R\eqdef  \lim_{n\to\infty}\int_{-M}^M dx~ \eta_x(x,t)\int_{-n}^{-R}d\al~ \frac{f^\ep(x-\al)}{\phi^\ep(\al)},
$$
$$
K_R\eqdef  \lim_{n\to\infty}\int_{-M}^M dx~ \eta_x(x,t)\int_{R}^{n}d\al~ \frac{f^\ep(x-\al)}{\phi^\ep(\al)}.
$$
We shall show how to control $J_R$, the same follows for $K_R$. We write
$$
J_R=\lim_{n\to\infty}\int_{-M}^M dx ~\eta_x(x,t)\int_{x+R}^{n}d\al~ \frac{f^\ep(\al)}{\phi^\ep(x-\al)}.
$$
An integration by parts yields
$$
J_R=\lim_{n\to\infty}\int_{-M}^M dx ~\eta(x,t)\left(\frac{f(x+R)|R|^\ep}{-R}+(1-\ep)\int_{x+R}^{n}d\al~ \frac{f^\ep(\al)}{|x-\al|^{2-\ep}} \right).
$$
Hence
$$
|J_R|\leq 2\|\eta\|_{L^1}\|f_0\|_{L^\infty}/R^{1/2}.
$$
Since the same estimate holds for $|K_R|$, one finds that $I_R$ is arbitrarily small if $R$ is arbitrarily large.

It remains to prove the convergence of
$$
\int_0^T\int_{B_R - B_\delta} d\al ~\eta_x(x,t)~ \arctan\left(\frac{f^\ep(x)-f^\ep(x-\al)}{\phi^\ep(\al)}\right).
$$
Recall that we have uniform convergence on compact sets. Lets consider 
$$
G^\ep = \frac{f^\ep(x) - f^\ep(x-\alpha) }{\phi^\ep(\alpha)},
$$
where $x\in B_M$ and $\alpha \in B_R - B_\delta$. Since $\arctan$ is a continuous function then $\arctan (G^\ep) \to \arctan (G^0)$ converges uniformly.  Thus also the integral over a bounded region of $\arctan (G^\ep)$ also converges.  Then for any $R> M +1$ and any small  $\delta >0$ as $\ep \downarrow 0$ we have
\begin{multline*}
\int_0^T ~ dt ~
\int_{\R} ~ dx ~
\eta_x(x,t)\frac{\rho}{\pi}\!\!\int_{B_R - B_\delta}~ d\al~ \arctan\left(\frac{f^\ep(x)-f^\ep(x-\al)}{\phi^\ep(\al)}\right)
\\
\to
\int_0^T ~ dt ~
\int_{\R} ~ dx ~
\eta_x(x,t)\frac{\rho}{\pi}\!\!\int_{B_R - B_\delta}~ d\al~ \arctan\left(\frac{f(x)-f(x-\al)}{\al}\right).
\end{multline*}
We conclude by first choosing $R$ sufficiently large and $\delta>0$ sufficiently small and then sending $\ep \downarrow 0$.  Note that $R$ and $\delta$ will generally depend upon the size of $\| f_0\|_\infty$, but this has no effect on our argument.

\subsection{Strong convergence in $L^\infty([0,T];L^\infty(B_R))$}
\label{sec:strong}

In order to prove the strong convergence in $L^\infty([0,T];L^\infty(B_R))$, the idea is to use the non-standard weak space $W_*^{-2,\infty}(B_R)$ which will be defined below. Crucially, we will have the uniform bounds:
\begin{equation}
\begin{split}
\sup_{t\in [0,T]}\| f^\ep (t)\|_{W^{1,\infty}(B_R)}\leq C \| f_0 \|_{L^\infty(\R)},
\\
\sup_{t\in [0,T]}\left\| \frac{\partial f^\ep}{\partial t} (t)\right\|_{ W_*^{-2,\infty}(B_{R})}
\le C  \| f_0 \|_{L^\infty(\R)},
\end{split}
\label{timeBOUND}
\end{equation}
where $C$ does not depend on $R$ or $\ep$. From here we will conclude that for any finite $R>0$ there exists a subsequence such that $f^\ep \to f$ strongly in $L^\infty([0,T];L^\infty(B_R))$. 

We define the space $W_*^{-2,\infty}(B_{R})$ as follows.  For $v\in L^\infty(B_{R})$ we consider the norm
$$
\| v\|_{-2,\infty}=  \sup_{\phi \in W^{2,1}_0(B_{R})\,:\, \| \phi \|_{2,1} \le 1} \left| \int_{B_R}\phi(x) v(x) dx \right|.
$$
Here $W^{2,1}_0(B_{R})$ is the usual set of functions in $W^{2,1}(B_{R})$ which vanish on the boundary of $B_{R}$ together with their first two weak derivatives.  Now the Banach space  $W_*^{-2,\infty}(B_{R})$ is defined to be the completion of $L^\infty(B_{R})$ with respect to the norm $\| \cdot \|_{-2,\infty}$. In general this is all we need for our convergence study.  This will be explained after the proof of Lemma \ref{convergence} below. The full space $W_*^{-2,\infty}$ may be a large space, but because we are going to deal with $f^\ep$ and $\frac{df^\ep}{dt}$, both in $L^\infty([0,T];L^\infty(B_R))$, it is not difficult to find the norms of both functions in $W_*^{-2,\infty}(B_{R})$.

These spaces are suitable because
$$
W^{1,\infty}(B_R)\subset L^\infty(B_R)\subset W_*^{-2,\infty}(B_{R}).
$$
Now the embedding $L^\infty(B_{R})\subset W_*^{-2,\infty}(B_{R})$ is continuous, and the embedding
$W^{1,\infty}(B_R)\subset L^\infty(B_R)$ is compact by the Arzela-Ascoli theorem.

We now proceed to discuss the convergence argument.
Arguments related to Lemma \ref{convergence} below are described for instance in \cite{CF}. However in \cite{CF} reflexive Banach spaces are used. None of thespaces used here are reflexive.

\begin{lemma}\label{convergence}
Consider a sequence $\{u_m\}$ in $C([0,T]\times B_R)$ that is uniformly bounded in the space
$L^\infty ([0,T]; W^{1,\infty}(B_R))$.
Assume further that the weak derivative $\frac{d u_m}{dt}$ is in $L^\infty([0,T];L^\infty(B_R))$ (not necessarily uniform) and is uniformly bounded in
$L^\infty ([0,T];W_*^{-2,\infty}(B_{R}))$. Finally suppose that $\partial_x u_m\in C([0,T]\times B_R)$. Then there exists a subsequence of $u_m$ that converges strongly in $L^\infty ([0,T];L^\infty(B_R))$.
\end{lemma}

\begin{proof}
Notice that it is enough to prove that the convergence is strong in the space $L^\infty ([0,T]; W_*^{-2,\infty}(B_{R}))$ because of the following interpolation theorem: for any small $\eta >0$ there exists $C_\eta >0$ such that
$$
\| u \|_{L^\infty}
\le
\eta \| u \|_{1,\infty}
+
C_\eta \| u \|_{-2,\infty}.
$$
This holds for all $u \in W^{1,\infty}(B_R).$ See, for example, \cite[Lemma 8.3]{CF}. Here we can replace reflexivity with the Banach-Alaoglu theorem in $W^{1,\infty}(B_R)$.

Let $t,s\in [0,T]$ be otherwise arbitrary.  We have
$$
u_m(t) - u_m(s) = \int_s^t d\tau ~   \frac{\partial u_m}{\partial \tau }(\tau).
$$
This holds rigorously in the sense that

\begin{equation}
\int_{B_R} u_m(t)\phi dx - \int_{B_R} u_m(s)\phi dx =
\int_s^t d\tau \int_{B_R}\frac{\partial u_m}{\partial \tau }(\tau) \phi dx.
\label{difference}
\end{equation}
for any $\phi \in W^{2,1}(B_R)$. Clearly, we find
\begin{equation*}
\| u_m(t) - u_m(s)  \|_{W_*^{-2,\infty}(B_{R})}\leq \sup_{\tau\in[0,T]} \left\| \frac{\partial u_m}{\partial \tau }(\tau) \right\|_{W_*^{-2,\infty}(B_{R})} | t - s|,
\end{equation*}
and therefore
\begin{equation}\label{noseque}
\| u_m(t) - u_m(s)  \|_{W_*^{-2,\infty}(B_{R})}\leq  L | t - s|,
\end{equation}
where
$$
L=\sup_{m\in\N}\sup_{\tau\in[0,T]} \left\| \frac{\partial u_m}{\partial \tau }(\tau) \right\|_{W_*^{-2,\infty}(B_{R})}.
$$
Now we consider $\{t_k\}_{k\in\N}=[0,T]\cap \Q$. We have $u_m(t_k)\in W^{1,\infty}(B_R)$ for any $m$ and $k$. By the standard diagonalization argument, we can get a subsequence (still denoted by $m$) such that
$$
u_m(t_k)\to u(t_k),
$$
for any $k$ in $L^\infty(B_R)$ as in the Arzela-Ascoli theorem.

Consider $\epsilon>0$. Since $[0,T]$ is compact, there exists $J\in\N$ such that
$$
[0,T]\subset \bigcup_{j=1}^J \left(t_{k_j}-\frac{\epsilon}{6L},t_{k_j}+\frac{\epsilon}{6L}\right).
$$
Then there exists $N_j$ such that $\forall$ $m_1,\,m_2\geq N_j$ it holds that
$$\|u_{m_1}(t_{k_j})-u_{m_2}(t_{k_j})\|_{W_*^{-2,\infty}(B_{R})}<\epsilon/3.$$
Taking $N=\D\max_{j=1,\ldots,J} N_j$, $\forall$ $m_1,\,m_2\geq N$ it is easy to check that
$$
\sup_{t\in [0,T]}\|u_{m_1}(t)-u_{m_2}(t)\|_{W_*^{-2,\infty}(B_{R})}<\epsilon.
$$
We find the sequence uniformly Cauchy in $L^\infty ([0,T];W_*^{-2,\infty}(B_{R}))$
which therefore converges strongly to an element in $L^\infty ([0,T];W_*^{-2,\infty}(B_{R}))$.
\end{proof}

Now we apply the above Lemma \ref{convergence} to prove the strong convergence which was {\it claimed} in the previous Section \ref{secweakSOL}. It remains only to prove that for any solution $f^\ep$ to \eqref{regularizedM} we have
$\frac{\partial f^\ep}{\partial t}$ in $L^\infty([0,T];L^\infty(B_R))$ (but not uniformly) and that the second inequality in \eqref{timeBOUND} holds for all sufficiently small $\ep >0$ and for any $R>0$.

Recall that $f^\ep\in C([0,T];H^3(\R))$, then in \eqref{regularizedM} the linear terms are bounded easily. The nonlinear term can be written as
$$
NL=-(1-\ep)C\Lambda^{1-\ep}f^\ep-\int_\R \frac{f^\ep_x(x)-f^\ep_x(x-\al)}{\phi^\ep(\al)}\frac{(\Delta_\al^\ep f^\ep(x))^2}{1+(\Delta_\al^\ep f^\ep(x))^2}d\al,
$$
and therefore
$$
|NL(x,t)|\leq C(\ep)\|f^\ep\|_{H^3}(t),
$$
by Sobolev embedding.

The norm of $\frac{\partial f^\ep}{\partial t} \in W_*^{-2,\infty}(B_R)$ is given by
$$
\left\| \frac{\partial f^\ep}{\partial t} (t)\right\|_{W_*^{-2,\infty}(B_R)}
 = \sup_{\phi \in W_0^{2,1}(B_R): \| \phi \|_{W^{2,1}} \le 1} \left| \int_{\R}  dx ~ \frac{\partial f^\ep}{\partial t}(x,t) \phi(x) \right|.
$$
Since $\phi$ vanishes on the boundary of $B_R$, we can think of $\phi(x)$ as being zero outside of the ball of radius $R$.  Then we are allowed to integrate over the whole space $\R$.  This is also unimportant below because the functions we are estimating are defined on the whole space $\R$.  It is however important because we want to estimate the non-local operator $\Lambda^{1-\ep}$ in this norm via ``integration by parts". Then we have
$$
I=\int_{B_R} \Lambda^{1-\ep}f(x) \phi(x)dx=\int_{\R} \Lambda^{1-\ep}f(x) \phi(x)dx=\int_{\R} f(x) \Lambda^{1-\ep}\phi(x)dx,
$$
and therefore
$$
|I|\leq \|f\|_{L^\infty}(t)\|\Lambda^{1-\ep}\phi\|_{L^1}.
$$
We compute
$$
\Lambda^{1-\ep}\phi(x)=c\int_{\R}\frac{\phi(x)-\phi(x-\al)}{|\al|^{2-\ep}}d\al=\int_{|\al|>1}\!\!d\al+\int_{|\al|<1}\!\!d\al=J_1(x)+J_2(x),
$$
thus
$$
\int_{\R}|J_1(x)|dx\leq \int_{|\al|>1}\frac{d\al}{|\al|^{2-\ep}}\int_\R dx(|\phi(x)|+|\phi(x-\al)|)\leq C\|\phi\|_{L^1(B_R)}.
$$
It is easy to rewrite $J_2$ as follows
$$
J_2(x)=c\int_{|\al|<1}\frac{\phi(x)-\phi(x-\al)-\phi_x(x)\al}{|\al|^{2-\ep}}d\al,
$$
and therefore the following identities
$$
\phi(x)-\phi(x-\al)=\al\int_0^1 \phi_x(x+(s-1)\al)ds,
$$
$$
\phi(x)-\phi(x-\al)-\phi_x(x)\al=\al^2\int_0^1(s-1)ds\int_0^1dr\, \phi_{xx}(x+r(s-1)\al),
$$
allow us to find
$$
\int_{\R}|J_1(x)|dx\leq \int_{|\al|<1}\int_0^1ds\int_0^1dr\int_{\R}dx ~ |\phi_{xx}(x+r(s-1)\al)|\leq 2\|\phi_{xx}\|_{L^1(B_R)}.
$$
Now we clearly have using the a priori bounds that
$$
\left\| \Lambda^{1-\ep}f^\ep \right\|_{W_*^{-2,\infty}(B_R)}
+
\left\| f^\ep_{xx} \right\|_{W_*^{-2,\infty}(B_R)}
\le
C \left\| f^\ep \right\|_{L^{\infty}(\R)}
\le
C \left\| f_0 \right\|_{L^{\infty}(\R)}.
$$
Here the constant is independent of $\ep$ and $R>0$.

For the last term in \eqref{regularizedM} we consider the duality relation
$$
\int_{\R}  dx ~  \partial_x u(x) ~  PV\int_{\R}d\al ~ \arctan(\Delta_\al^\ep f^\ep(x)).
$$
Using exactly the arguments from Section \ref{secweakSOL} with say $R=\delta=1$ we have 
$$
\left|
\int_{\R}  dx ~  \partial_x u(x) ~  PV\int_{\R}d\al ~ \arctan(\Delta_\al^\ep f(x))
\right|
\le
C\| u \|_{W^{1,1}}
 \left\| f_0 \right\|_{L^{\infty}(\R)}.
$$
We thus conclude \eqref{timeBOUND}.  \hfill {\bf Q.E.D.}

\subsection*{{\bf Acknowledgments}}

\smallskip

PC was partially supported by NSF grant DMS-0804380.
DC and FG were partially supported by MCINN grant MTM2008-03754 (Spain) and ERC grant StG-203138CDSIF. FG was partially supported by NSF grant DMS-0901810. RMS was partially supported by NSF grant DMS-0901463.

\begin{quote}
\begin{tabular}{l}
\textbf{Peter Constantin }\\
{\small Department of Mathematics}\\
{\small University of Chicago}\\
{\small 5734 University Avenue, Chicago, IL 60637, USA}\\
{\small Email: const@cs.uchicago.edu}
\end{tabular}
\end{quote}

%\begin{quote}
%\begin{tabular}{ll}
%\textbf{Diego C\'ordoba} &  \textbf{Francisco Gancedo}\\
%{\small Instituto de Ciencias Matem\'aticas} & {\small Department of Mathematics}\\
%{\small Consejo Superior de Investigaciones Cient\'ificas} & {\small University of Chicago}\\
%{\small Serrano 123, 28006 Madrid, Spain} & {\small 5734 University Avenue, Chicago, IL 60637}\\
%{\small Email: dcg@icmat.es} & {\small Email: fgancedo@math.uchicago.edu}
%\end{tabular}
%\end{quote}

\begin{quote}
\begin{tabular}{ll}
\textbf{Diego C\'ordoba}\\
{\small Instituto de Ciencias Matem\'aticas}\\
{\small Consejo Superior de Investigaciones Cient\'ificas}\\
{\small Serrano 123, 28006 Madrid, Spain}\\
{\small Email: dcg@icmat.es}
\end{tabular}
\end{quote}

\begin{quote}
\begin{tabular}{ll}
\textbf{Francisco Gancedo}\\
{\small Department of Mathematics}\\
{\small University of Chicago}\\
{\small 5734 University Avenue, Chicago, IL 60637, USA}\\
{\small Email: fgancedo@math.uchicago.edu}
\end{tabular}
\end{quote}

\begin{quote}
\begin{tabular}{ll}
\textbf{Robert M. Strain}\\
{\small Department of Mathematics}\\
{\small University of Pennsylvania}\\
{\small David Rittenhouse Lab}\\
{\small 209 South 33rd Street, Philadelphia, PA 19104, USA} \\ %-6395, USA}\\
{\small Email: strain@math.upenn.edu}
\end{tabular}
\end{quote}


\begin{thebibliography}{99}

\bibitem{Am} D. Ambrose. Well-posedness of Two-phase Hele-Shaw Flow without Surface Tension. Euro. Jnl. of Applied Mathematics 15 597-607, 2004.

\bibitem{bear} J. Bear, Dynamics of Fluids in Porous Media, \emph{American
Elsevier}, New York, 1972.

\bibitem{B-C} A.~L. Bertozzi and P. Constantin. Global regularity for vortex patches.
\emph{Comm. Math. Phys.} 152 (1): 19--28, 1993.

\bibitem{CF} P. Constantin and C. Foias. Navier-Stokes Equation. \emph{Chicago Lectures in Mathematics}.

\bibitem{Peter} P. Constantin and M. Pugh. Global solutions for small data to the
Hele-Shaw problem. \emph{Nonlinearity}, 6 (1993), 393 - 415.

\bibitem{ADP} A. C\'ordoba, D. C\'ordoba and F. Gancedo. Interface evolution: the Hele-Shaw and Muskat problems. \emph{To appear in Annals of Math.}

\bibitem{DP} D. C\'ordoba and F. Gancedo. Contour dynamics of incompressible 3-D fluids in a porous medium with different densities. \emph{Comm. Math. Phys.}, 273  (2007),  no. 2, 445-471.

\bibitem{DP2} D. C\'ordoba and F. Gancedo. A maximum principle for the Muskat problem for fluids with different densities. \emph{Comm. Math. Phys.}, 286 (2009), no. 2, 681-696.

\bibitem{DPR} D. C\'ordoba, F. Gancedo and R. Orive. A note on the interface dynamics for convection in porous media. \emph{Physica D}, 237 (2008), 1488-1497.

\bibitem{DP3} D. C\'ordoba and F. Gancedo. Absence of squirt singularities for the multi-phase Muskat problem. \emph{To appear in Comm. Math. Phys.}

\bibitem{Esch1} J. Escher and G. Simonett. Classical solutions for Hele-Shaw models with surface tension. Adv. Differential Equations, 2:619-642, 1997.

\bibitem{Esch2} J. Escher and B.V. Matioc. On the parabolicity of the Muskat problem: Well-posedness, fingering, and stability results. ArXiv:1005.2512.

\bibitem{Esch3} J. Escher, A.-V. Matioc and B.V. Matioc: A generalised Rayleigh-Taylor condition for the Muskat problem. Arxiv: 1005.2511.

\bibitem{Muskat} M. Muskat. \newblock The flow of homogeneous fluids through porous media. \newblock \emph{New York}, 1937.

\bibitem{S-T} P.G. Saffman and Taylor. The penetration of a fluid into a porous medium or Hele-Shaw cell containing a more viscous liquid. \newblock \emph{Proc. R. Soc. London, Ser. A} 245, 312-329, 1958.

\bibitem{SCH} M. Siegel, R. Caflisch and S. Howison. Global
Existence, Singular Solutions, and Ill-Posedness for the Muskat Problem. \emph{Comm. Pure and Appl.
Math.}, 57: 1374-1411, 2004.

\bibitem{St3} E.~Stein. \newblock Harmonic Analysis. \newblock \emph{%
Princeton University Press.} Princeton, NJ, 1993.

\bibitem{Yi} F. Yi. Local classical solution of Muskat free boundary problem, J. Partial Diff. Eqs., 9 (1996), 8496.

\bibitem{Yi2} F. Yi. Global classical solution of Muskat free boundary problem, J. Math. Anal. Appl., 288 (2003), 442461.

\end{thebibliography}
\end{document}